\documentclass[journal,onecolumn,12pt]{article}

\usepackage{amssymb}
\usepackage{caption2}
\usepackage{amsmath}
\usepackage{multirow}
\usepackage{amsthm}
\usepackage{graphicx}
\usepackage{CJK}
\usepackage{enumerate}

\usepackage{tikz}
\usepackage{tikz}
\usetikzlibrary{positioning,chains,fit,shapes,calc}

\newtheorem{theorem}{Theorem}[section]
\newtheorem{lemma}[theorem]{Lemma}
\newtheorem{proposition}[theorem]{Proposition}

\begin{document}
\title{Some sum-product estimates in matrix rings over finite fields}

\author{Chengfei Xie$^{\text{a}}$ and Gennian Ge$^{\text{a,}}$\thanks{%Corresponding author. Email address:  gnge@zju.edu.cn (G. Ge).
  Corresponding author. Email address: gnge@zju.edu.cn. Research supported by the National Natural Science Foundation of
China under Grant No. 11971325, National Key Research and Development Program of China under Grant Nos. 2020YFA0712100 and 2018YFA0704703,
and Beijing Scholars Program.}\\
\footnotesize $^{\text{a}}$ School of Mathematical Sciences, Capital Normal University, Beijing, 100048, China}

\maketitle

\begin{abstract}
We study some sum-product problems over matrix rings. Firstly, for $A, B, C\subseteq M_n(\mathbb{F}_q)$, we have
$$
|A+BC|\gtrsim q^{n^2},
$$
whenever $|A||B||C|\gtrsim q^{3n^2-\frac{n+1}{2}}$. Secondly, if a set $A$ in $M_n(\mathbb{F}_q)$ satisfies $|A|\geq C(n)q^{n^2-1}$ for some sufficiently large $C(n)$, then we have
$$
\max\{|A+A|, |AA|\}\gtrsim \min\left\{\frac{|A|^2}{q^{n^2-\frac{n+1}{4}}}, q^{n^2/3}|A|^{2/3}\right\}.
$$
These improve the results due to The and Vinh (2020), and generalize the results due to Mohammadi, Pham, and Wang (2021). We also give a new proof for a recent result due to The and Vinh (2020). Our method is based on spectral graph theory and linear algebra.
\smallskip
\end{abstract}
\medskip


\smallskip

%\begin{keyword}
%Locally repairable codes, sequential recovery,  multiple erasures, distributed storage systems
%\end{keyword}

\section{Introduction}
Let $\mathbb{F}_q$ be a field with $q$ elements. Let $M_n(\mathbb{F}_q)$ be the ring of all $n\times n$ matrices over $\mathbb{F}_q$, $Z_n(\mathbb{F}_q)$ be the set of $n\times n$ matrices over $\mathbb{F}_q$ with zero determinant, and $GL_n(\mathbb{F}_q)$ be the set of $n\times n$ invertible matrices over $\mathbb{F}_q$. Throughout the paper, we write $X\lesssim Y$ if there exists a constant $C(n)$ (maybe dependent on $n$, but independent with $q$) such that $X\leq C(n)Y$, and write $X\sim Y$ if $X\lesssim Y$ and $Y\lesssim X$. For $A, B\subseteq M_n(\mathbb{F}_q)$, we define $A+B=\{a+b:a\in A, b\in B \}$, $AB=\{ab:a\in A, b\in B\}$, $-A=\{-a:a\in A\}$. If $A\subseteq GL_n(\mathbb{F}_q)$, then we write $A^{-1}=\{a^{-1}:a\in A\}$. Moreover, let $I_n$ be the $n\times n$ identity matrix.

In an arbitrary ring $R$, the sum-product problem asks the lower bound of $\max\{|A+A|, |AA|\}$ for $A\subseteq R$ under some conditions. In \cite{MR820223}, Erd\H{o}s and Szemer\'{e}di proved that there exists a constant $\epsilon$ such that
$$
\max\{|A+A|, |AA|\}\gtrsim|A|^{1+\epsilon},
$$
for any finite set $A\subseteq\mathbb{Z}$. They also conjectured that this bound holds for any $\epsilon<1$ and any sufficiently large $A$.

In \cite{MR1472816}, Elekes introduced a geometric approach (namely, the Szemer\'{e}di-Trotter theorem) for the sum-product problem, and obtained that
$$
\max\{|A+A|, |AA|\}\gtrsim|A|^{5/4},
$$
for any finite set $A\subseteq\mathbb{R}$. It shows the relationship between the sum-product problem and incidence geometry. The best known bound in this direction is due to Shakan \cite{MR4015652}, which states that
$$
\max\{|A+A|, |AA|\}\gtrsim|A|^{\frac{4}{3}+\frac{5}{5277}}.
$$

In the setting of finite fields, surprising results are obtained when $A\subseteq\mathbb{F}_q$ is large with respect to $q$. In particular, when $A=\mathbb{F}_q$, then $|A+A|=|AA|=|A|=q$. So one generally works either on the bound of $\max\{|A+A|, |AA|\}$ when $A$ is small in terms of the characteristic $p$ of $\mathbb{F}_q$ or on the lower the size of $|A|$ to guarantee that $\max\{|A+A|, |AA|\}$ is large in terms of $q$. Bourgain, Katz, and Tao \cite{MR2053599} showed that, given $A\subseteq\mathbb{F}_p$ with $p$ prime and $p^\delta<|A|<^{1-\delta}$ for some $\delta>0$, one has
$$
\max\{|A+A|, |AA|\}\geq C_{\delta}|A|^{1+\epsilon},
$$
for some $\epsilon=\epsilon(\delta)$. Notably, Roche-Newton, Rudnev, and Shkredov \cite{MR3474329} showed that
$$
\max\{|A+A|, |AA|\}\gtrsim|A|^{6/5},
$$
for $A\subseteq\mathbb{F}_q$ with characteristic $p$ and $|A|<p^{5/8}$. Rudnev, Shakan, and Shkredov \cite{MR4069186} improved the exponent to $11/9$ for $A\subseteq\mathbb{F}_p^*$ and $|A|<p^{36/67}$. Most recently, Mohammadi and Stevens \cite{2021arXiv210308252M} improved the exponent from $11/9$ to $5/4$ for $A\subseteq\mathbb{F}_p$ and $|A|\lesssim p^{1/2}$.

In matrix rings, Karabulut, Koh, Pham, Shen, and Vinh \cite{MR3975670} proved the following result.
\begin{theorem}[\cite{MR3975670}]
If $A\subseteq M_2(\mathbb{F}_q)$ with $|A|\geq Cq^{3}$ for some constant $C$, then we have
$$
\max\{|A+A|, |AA|\}\gtrsim \min\left\{\frac{|A|^2}{q^{7/2}}, q^{2}|A|^{1/2}\right\}.
$$
\end{theorem}
Some other results were obtained as well. Their work was generalized by The and Vinh \cite{MR4130079}.

\begin{theorem}[\cite{MR4130079}]\label{A+AAA}
For every positive integer $n$, there exists $C(n)$ such that the following holds. If $A\subseteq M_n(\mathbb{F}_q)$ with $|A|\geq C(n)q^{n^2-1}$, then we have
$$
\max\{|A+A|, |AA|\}\gtrsim \min\left\{\frac{|A|^2}{q^{n^2-1/2}}, q^{n^2/2}|A|^{1/2}\right\}.
$$
\end{theorem}

\begin{theorem}[\cite{MR4130079}]\label{A+BC}
For $A, B, C\subseteq M_n(\mathbb{F}_q)$, we have
$$
|A+BC|\gtrsim\min\left\{q^{n^2}, \frac{|A||B||C|}{q^{2n^2-1}}\right\}.
$$
\end{theorem}

\begin{theorem}[\cite{MR4130079}]\label{(A+B)C}
For $A, B\subseteq M_n(\mathbb{F}_q)$ and $C\subseteq GL_n(\mathbb{F}_q)$, we have
$$
|(A+B)C|\gtrsim\min\left\{q^{n^2}, \frac{|A||B||C|}{q^{2n^2-1}}\right\}.
$$
\end{theorem}
We refer the readers to \cite{2021arXiv210607328M}, \cite{MR4031446} and \cite{2021arXiv210603319V} for related results.

In this paper, we give some new results of sum-product estimates, which are also generalizations of the results in \cite{2021arXiv210607328M}.
\begin{theorem}\label{A+BC2}
For $A, B, C\subseteq M_n(\mathbb{F}_q)$, we have
$$
|A+BC|\gtrsim\min\left\{q^{n^2}, \frac{|A||B||C|}{q^{2n^2-\frac{n+1}{2}}}\right\}.
$$
In particular, if $|A||B||C|\gtrsim q^{3n^2-\frac{n+1}{2}}$, then $|A+BC|\gtrsim q^{n^2}$.
\end{theorem}

\begin{theorem}\label{A+AAA2}
For every positive integer $n$, there exists $C(n)$ such that the following holds. If $A\subseteq M_n(\mathbb{F}_q)$ with $|A|\geq C(n)q^{n^2-1}$, we have
$$
\max\{|A+A|, |AA|\}\gtrsim \min\left\{\frac{|A|^2}{q^{n^2-\frac{n+1}{4}}}, q^{n^2/3}|A|^{2/3}\right\}.
$$
\end{theorem}

Observe that Theorem \ref{A+BC2} is better than Theorem \ref{A+BC}. And Theorem \ref{A+AAA2} is better than Theorem \ref{A+AAA} in some cases. For example, put $n=4$ and $|A|\sim q^{15.01}$. Then Theorem \ref{A+AAA} gives that
$$
\max\{|A+A|, |AA|\}\gtrsim q^{14.52},
$$
while Theorem \ref{A+AAA2} gives that
$$
\max\{|A+A|, |AA|\}\gtrsim q^{15.27}.
$$

Finally, we will give a new proof of Theorem \ref{(A+B)C}.

\section{Preliminaries}
Let $G=(U\cup V, E)$ be a biregular graph. We write $\deg(U)$ for the common degree of vertices in $U$. Let $A_G$ be the adjacency matrix of $G$, and suppose that $|\lambda_1|\geq|\lambda_2|\geq|\lambda_3|\cdots\geq|\lambda_n|$ are eigenvalues of $A_G$. Note that in a bipartite graph, we have $\lambda_1=-\lambda_2$. We call $\lambda_3$ the third eigenvalue of $G$ and we need the following lemma, which is a variant of the expander mixing lemma.
\begin{lemma}[\cite{MR3937692}]\label{expander}
Let $G$ be a biregular graph with parts $U$ and $V$. Then, for every pair $X\subseteq U$ and $Y\subseteq V$, the number of edges between $X$ and $Y$, denoted by $e(X, Y)$, satisfies
$$
\left|e(X, Y)-\frac{\deg(U)}{|V|}|X||Y|\right|\leq|\lambda_3|\sqrt{|X||Y|},
$$
where $\lambda_3$ is the third eigenvalue of $G$.
\end{lemma}

\begin{lemma}[\cite{MR3937692}]\label{NNT}
Let $G$ be a biregular graph with parts $U$ and $V$, and $|U| = m, |V| = n$. We label vertices of $G$ from $1$ to $|U| + |V|$. Let $A_G$ be the adjacency matrix of $G$ having the form
$$
A_G=\left(
  \begin{array}{cc}
    0 & N \\
    N^T & 0 \\
  \end{array}
\right),
$$
where $N$ is the $|U| \times |V|$ matrix, and $N_{ij} = 1$ if and only if there is an edge between $i$ and $j$. Let $v^3 = (u_1, \ldots , u_m, v_1, \ldots , v_n)^T$ be an eigenvector of $A_G$ corresponding to the eigenvalue $\lambda_3$. Then we have

(i) $(u_1, \ldots , u_m)^T$ is an eigenvector of $NN^T$, and

(ii) $J(u_1, \ldots, u_m)^T = 0$, where $J$ is the $m \times m$ all-ones matrix.
\end{lemma}

\section{A key lemma}\label{keylemma}
Given sets $A, B, C, D, E, F\subseteq M_n(\mathbb{F}_q)$, let $N(A, B, C, D, E, F)$ be the number of solutions to the equation
\begin{equation}
  ab+ef=c+d,\quad(a, b, c, d, e, f)\in A\times B\times C\times D\times E\times F.
\end{equation}
We have the following proposition.
\begin{proposition}\label{main2}
For every positive integer $n$, there exists $C(n)$ such that the following holds. For $A, B, C, D, E, F\subseteq M_n(\mathbb{F}_q)$, we have
$$
N(A, B, C, D, E, F)\leq C(n)\left(\frac{|A||B||C||D||E||F|}{q^{n^2}}+q^{2n^2-\frac{n+1}{2}}\sqrt{|A||B||C||D||E||F|}\right).
$$
\end{proposition}
\begin{proof}
We construct a graph $G=(U\cup V, E)$, where $U=V=(M_n(\mathbb{F}_q))^3$. There is an edge between $(a, e, c)\in U$ and $(b, f, d)\in V$ if and only if $ab+ef=c+d$. It is easy to check that
$$
|U|=|V|=(|M_n(\mathbb{F}_q)|)^3=q^{3n^2}.
$$
Given $(a, e, c)\in U$ and $(b, f)\in(M_n(\mathbb{F}_q))^2$, $d=ab+ef-c$ is uniquely determined. So the number of neighbors of $(a, e, c)\in U$ in the graph $G$ is $\deg(U)=q^{2n^2}$. And
$$
\frac{\deg(U)}{|V|}=\frac{1}{q^{n^2}}.
$$
Similarly, the number of neighbors of $(b, f, d)\in V$ in the graph $G$ is $q^{2n^2}$ too.

For any two points $(a_1, e_1, c_1)$ and $(a_2, e_2, c_2)$ in $U$, we count the number of their common neighbors, i.e., the number of solutions $(b, f, d)$ to the equations
\begin{equation}\label{7}
  a_1b+e_1f=c_1+d, \quad a_2b+e_2f=c_2+d.
\end{equation}
So we have
\begin{equation}
(a_1-a_2)b+(e_1-e_2)f=c_1-c_2,
\end{equation}
or equivalently,
\begin{equation}\label{8}
  \left(
    \begin{array}{cc}
      a_1-a_2 & e_1-e_2 \\
    \end{array}
  \right)
\left(
  \begin{array}{c}
    b \\
    f \\
  \end{array}
\right)=
c_1-c_2.
\end{equation}
A solution $\left(
              \begin{array}{c}
                b \\
                f \\
              \end{array}
            \right)$
to equation (\ref{8}) corresponds to a solution $(b, f, a_1b+e_1f-c_1)$ to equations (\ref{7}). So we only need to determine the number of solutions to equation (\ref{8}).

We need the following theorems in linear algebra.
\begin{theorem}
Let $A$ be a matrix of size $m\times n$. All of the solutions to the equation $AX=0$ form a vector space of dimension $n-\mathrm{rank}(A)$.
\end{theorem}
\begin{theorem}
Let $A$ be a matrix of size $m\times n$, and $b$ be a matrix of size $m\times 1$. Then the equation $AX=b$ has a solution if and only if
$$
\mathrm{rank}(A)=\mathrm{rank}\left(
              \begin{array}{cc}
                A & b \\
              \end{array}
            \right).
$$
Once $AX=b$ has a solution $X_0$, then every solution can be written as $X=X_0+X_1$, where $X_1$ is any solution to $AX=0$.
\end{theorem}

Using these theorems, we see that equation (\ref{8}) has a solution if and only if
$$
 \mathrm{rank} \left(
    \begin{array}{cc}
      a_1-a_2 & e_1-e_2 \\
    \end{array}
  \right)=
  \mathrm{rank}\left(
    \begin{array}{ccc}
      a_1-a_2 & e_1-e_2&c_1-c_2 \\
    \end{array}
  \right).
$$
And once equation (\ref{8}) has a solution, the number of solutions $\left(
              \begin{array}{c}
                b \\
                f \\
              \end{array}
            \right)$
is equal to $q^{(2n-k)n}$, where $k$ is the rank of $\left(
    \begin{array}{cc}
      a_1-a_2 & e_1-e_2 \\
    \end{array}
  \right)$, since each column of $\left(
              \begin{array}{c}
                b \\
                f \\
              \end{array}
            \right)$ has $q^{2n-k}$ choices.

If $k=\mathrm{rank} \left(
    \begin{array}{cc}
      a_1-a_2 & e_1-e_2 \\
    \end{array}
  \right)=0$, then $c_1-c_2$ must be $0$ to guarantee that
  $$
 \mathrm{rank} \left(
    \begin{array}{cc}
      a_1-a_2 & e_1-e_2 \\
    \end{array}
  \right)=
  \mathrm{rank}\left(
    \begin{array}{ccc}
      a_1-a_2 & e_1-e_2&c_1-c_2 \\
    \end{array}
  \right).
$$
Then $a_1=a_2, e_1=e_2, c_1=c_2$, which contradicts that $(a_1, e_1, c_1)$ and $(a_2, e_2, c_2)$ are different. So equation (\ref{8}) has no solution if $k=\mathrm{rank} \left(
    \begin{array}{cc}
      a_1-a_2 & e_1-e_2 \\
    \end{array}
  \right)=0$.

For $1\leq k\leq n$, let $E_{k}$ be the adjacency matrix of the graph $G_{k}$, whose vertex set is $(M_n(\mathbb{F}_q))^3$, such that two vertices $(a_1, e_1, c_1)$ and $(a_2, e_2, c_2)$ form an edge if  and only if
$$\mathrm{rank} \left(
    \begin{array}{cc}
      a_1-a_2 & e_1-e_2 \\
    \end{array}
  \right)=
  \mathrm{rank}\left(
    \begin{array}{ccc}
      a_1-a_2 & e_1-e_2&c_1-c_2 \\
    \end{array}
  \right)=k.
$$
If $(0,0, 0)$ is adjacent with $(a, e, c)$, then $(a', e', c')$ is adjacent with $(a+a', e+e', c+c')$, and vice versa. So $G_{k}$ is regular. We count the degree of $(0,0, 0)$, i.e., the number of $(a, e, c)$ with the property that
$$
\mathrm{rank} \left(
    \begin{array}{cc}
      a & e \\
    \end{array}
  \right)=
  \mathrm{rank}\left(
    \begin{array}{ccc}
      a & e&c \\
    \end{array}
  \right)=k.
$$
We first choose $\left(
    \begin{array}{cc}
      a & e \\
    \end{array}
  \right)$
such that $\mathrm{rank} \left(
    \begin{array}{cc}
      a & e \\
    \end{array}
  \right)=k$
and we need the following theorem.
\begin{theorem}[\cite{Landsberg1893}]\label{rank}
The number of matrices of size $m\times n$ and with rank $k$ over $\mathbb{F}_q$ is $\frac{Q_{k}(q^m)Q_k(q^n)}{Q_k(q^k)}$, where $Q_k(q^m)=(q^m-1)(q^m-q)\cdots(q^m-q^{k-1})$.
\end{theorem}
Since $\left(
    \begin{array}{cc}
      a & e \\
    \end{array}
  \right)$ is an $n\times 2n$ matrix, Theorem \ref{rank} implies that there are $\frac{Q_{k}(q^{2n})Q_{k}(q^n)}{Q_{k}(q^{k})}$ choices for $\left(
    \begin{array}{cc}
      a & e \\
    \end{array}
  \right)$. Next we choose $c$. Since $
\mathrm{rank} \left(
    \begin{array}{cc}
      a & e \\
    \end{array}
  \right)=
  \mathrm{rank}\left(
    \begin{array}{ccc}
      a & e&c \\
    \end{array}
  \right)=k,
$
it follows that every column of $c$ is in the column space of $\left(
    \begin{array}{cc}
      a & e \\
    \end{array}
  \right)$, and hence every column of $c$ has $q^k$ choices. So the number of $(a, e, c)$ with the property that
$$
\mathrm{rank} \left(
    \begin{array}{cc}
      a & e \\
    \end{array}
  \right)=
  \mathrm{rank}\left(
    \begin{array}{ccc}
      a & e&c \\
    \end{array}
  \right)=k
$$
is
$$
\frac{Q_{k}(q^{2n})Q_{k}(q^n)}{Q_{k}(q^{k})} q^{nk}\sim q^{4nk-k^2}.
$$

For $0\leq k\leq n-1$, let $F_{k}$ be the adjacency matrix of the graph $H_{k}$, whose vertex set is $(M_n(\mathbb{F}_q))^3$, such that two vertices $(a_1, e_1, c_1)$ and $(a_2, e_2, c_2)$ form an edge in $H_{k}$  if  and only if
$$\mathrm{rank} \left(
    \begin{array}{cc}
      a_1-a_2 & e_1-e_2 \\
    \end{array}
  \right)=k<
  \mathrm{rank}\left(
    \begin{array}{ccc}
      a_1-a_2 & e_1-e_2&c_1-c_2 \\
    \end{array}
  \right).
$$
If $(0,0, 0)$ is adjacent with $(a, e, c)$, then $(a', e', c')$ is adjacent with $(a+a', e+e', c+c')$, and vice versa. So $H_{k}$ is regular. We count the degree of $(0,0, 0)$, i.e., the number of $(a, e, c)$ with the property that
$$
\mathrm{rank} \left(
    \begin{array}{cc}
      a & e \\
    \end{array}
  \right)=k<
  \mathrm{rank}\left(
    \begin{array}{ccc}
      a & e&c \\
    \end{array}
  \right).
$$
We first choose $\left(
    \begin{array}{cc}
      a & e \\
    \end{array}
  \right)$
such that $\mathrm{rank} \left(
    \begin{array}{cc}
      a & e \\
    \end{array}
  \right)=k$. There are $\frac{Q_{k}(q^{2n})Q_{k}(q^n)}{Q_{k}(q^{k})}$ choices for $\left(
    \begin{array}{cc}
      a & e \\
    \end{array}
  \right)$. Next we choose $c$. The number of choices for $c$ such that $$
\mathrm{rank} \left(
    \begin{array}{cc}
      a & e \\
    \end{array}
  \right)=
  \mathrm{rank}\left(
    \begin{array}{ccc}
      a & e&c \\
    \end{array}
  \right)=k
$$
is $q^{nk}$, so the number of choices for $c$ such that $$
\mathrm{rank} \left(
    \begin{array}{cc}
      a & e \\
    \end{array}
  \right)=k<
  \mathrm{rank}\left(
    \begin{array}{ccc}
      a & e&c \\
    \end{array}
  \right)
$$
is $q^{n^2}-q^{nk}$.
Hence the number of $(a, e, c)$ with the property that
$$
\mathrm{rank} \left(
    \begin{array}{cc}
      a & e \\
    \end{array}
  \right)=k<
  \mathrm{rank}\left(
    \begin{array}{ccc}
      a & e&c \\
    \end{array}
  \right)
$$
is
$$
\frac{Q_{k}(q^{2n})Q_{k}(q^n)}{Q_{k}(q^{k})} (q^{n^2}-q^{nk})\sim q^{n^2+3nk-k^2}.
$$

Based on the previous calculation, we have
\begin{equation}\label{9}
\begin{split}
  NN^T  =& q^{n^2}J+(\deg(U)-q^{n^2})I+\sum_{k=1}^{n}(q^{2n^2-nk}-q^{n^2})E_{k}-\sum_{k=0}^{n-1}q^{n^2}F_{k}\\
  =&q^{n^2}J+(\deg(U)-q^{n^2})I+\sum_{k=1}^{n-1}(q^{2n^2-nk}-q^{n^2})E_{k}-\sum_{k=0}^{n-1}q^{n^2}F_{k},
\end{split}
\end{equation}
where $I$ is the identity matrix.

Let $v^3 = (u_1, \ldots , u_{|U|}, v_1, \ldots , v_{|V|})^T$ be an eigenvector of $A_G$ corresponding to the eigenvalue $\lambda_3$. Lemma \ref{NNT} implies that $(u_1, \ldots , u_{|U|})^T$ is an eigenvector of $NN^T$ corresponding to the eigenvalue $\lambda_3^2$. It follows from equation (\ref{9}) that
\begin{equation}\label{10}
  (\lambda_3^2-\deg(U)+q^{n^2})(u_1, \ldots , u_{|U|})^T  =  \left(\sum_{k=1}^{n-1}(q^{2n^2-nk}-q^{n^2})E_{k}-\sum_{k=0}^{n-1}q^{n^2}F_{k}\right)(u_1, \ldots , u_{|U|})^T.
\end{equation}
Therefore, $(u_1, \ldots , u_{|U|})^T$ is an eigenvector of
$$
\sum_{k=1}^{n-1}(q^{2n^2-nk}-q^{n^2})E_{k}-\sum_{k=0}^{n-1}q^{n^2}F_{k}
$$
corresponding to the eigenvalue $\lambda_3^2-\deg(U)+q^{n^2}$.

Since $G_{k}$ is regular, for every eigenvalue $\lambda$ of $E_{k}$, we have $|\lambda|\lesssim q^{4nk-k^2}$. Since $H_k$ is regular, for every eigenvalue $\lambda$ of $F_{k}$, we have $|\lambda|\lesssim q^{n^2+3nk-k^2}$. So if $\lambda$ is an eigenvalue of
$$
\sum_{k=1}^{n-1}(q^{2n^2-nk}-q^{n^2})E_{k}-\sum_{k=0}^{n-1}q^{n^2}F_{k},
$$
then
\begin{equation}
\begin{split}
|\lambda|&\lesssim\sum_{k=1}^{n-1}(q^{2n^2-nk}-q^{n^2})q^{4nk-k^2}+\sum_{k=0}^{n-1}q^{n^2}q^{n^2+3nk-k^2}\\
&\leq\sum_{k=1}^{n-1}q^{2n^2-nk}q^{4nk-k^2}+\sum_{k=0}^{n-1}q^{n^2}q^{n^2+3nk-k^2}\\
&\lesssim\sum_{k=0}^{n-1}q^{2n^2+3nk-k^2}.
\end{split}
\end{equation}

Observe that the function $f(k)=2n^2+3nk-k^2$ is increasing for $k\leq3n/2$, so the maximum occurs at $k=n-1$ and $2n^2+3nk-k^2\leq2n^2+3n(n-1)-(n-1)^2=4n^2-n-1$.
Therefore, the eigenvalue $\lambda_3^2-\deg(U)+q^{n^2}$ of
$$
\sum_{k=1}^{n-1}(q^{2n^2-nk}-q^{n^2})E_{k}-\sum_{k=0}^{n-1}q^{n^2}F_{k}
$$
satisfies that
$$
|\lambda_3^2-\deg(U)+q^{n^2}|\lesssim q^{4n^2-n-1}.
$$
Note that $\deg(U)=q^{2n^2}$. So we conclude that
$$
|\lambda_3|\lesssim q^{2n^2-\frac{n+1}{2}}.
$$

Now if $A, B, C, D, E, F\subseteq M_n(\mathbb{F}_q)$, then we can view $A\times E\times C$ as a subset of $U$ and $B\times F\times D$ as a subset of $V$, and $N(A, B, C, D, E, F)$ is equal to $e(A\times E\times C, B\times F\times D)$. So Lemma \ref{expander} shows that
$$
\begin{array}{lll}
N(A, B, C, D, E, F)&\leq&{\frac{\deg(U)}{|V|}|A\times E\times C||B\times F\times D|}+|\lambda_3|\sqrt{|A\times E\times C||B\times F\times D|}\\
&\leq&C(n)\left(\frac{|A||B||C||D||E||F|}{q^{n^2}}+q^{2n^2-\frac{n+1}{2}}\sqrt{|A||B||C||D||E||F|}\right).
\end{array}
$$

\end{proof}

\section{Proofs of Theorem \ref{A+BC2} and Theorem \ref{A+AAA2}}
In this section, we prove Theorem \ref{A+BC2} and Theorem \ref{A+AAA2}. We first prove Theorem \ref{A+BC2}. For convenience, we restate it here.
\begin{theorem}\label{A+BC3}
For $A, B, C\subseteq M_n(\mathbb{F}_q)$, we have
$$
|A+BC|\gtrsim\min\left\{q^{n^2}, \frac{|A||B||C|}{q^{2n^2-\frac{n+1}{2}}}\right\}.
$$
\end{theorem}
\begin{proof}
For $\lambda\in A+BC$, let
$$
t(\lambda)=|\{(a, b, c)\in A\times B\times C:a+bc=\lambda\}|.
$$
By the Cauchy-Schwarz inequality, we have
$$
(|A||B||C|)^2=\left(\sum_{\lambda\in A+BC}t(\lambda)\right)^2\leq|A+BC|\sum_{\lambda\in A+BC}t(\lambda)^2.
$$
Note that
$$
\sum_{\lambda\in A+BC}t(\lambda)^2=N(B, C, A, -A, -B, C).
$$
Proposition \ref{main2} implies that
$$
\frac{(|A||B||C|)^2}{|A+BC|}\leq N(B, C, A, -A, -B, C)\lesssim\frac{|A|^2|B|^2|C|^2}{q^{n^2}}+q^{2n^2-\frac{n+1}{2}}|A||B||C|.
$$
So
$$
\frac{(|A||B||C|)^2}{|A+BC|}\lesssim\frac{|A|^2|B|^2|C|^2}{q^{n^2}}
$$
or
$$
\frac{(|A||B||C|)^2}{|A+BC|}\lesssim q^{2n^2-\frac{n+1}{2}}|A||B||C|.
$$
We conclude that
$$
|A+BC|\gtrsim\min\left\{q^{n^2}, \frac{|A||B||C|}{q^{2n^2-\frac{n+1}{2}}}\right\}.
$$
\end{proof}

Before proving Theorem \ref{A+AAA2}, we need an estimate of additive energy.

For $A, B\subseteq M_n(\mathbb{F}_q)$, define
$$
E_+(A, B)  =|\{(a_1, a_2, b_1, b_2)\in A^2\times B^2:a_1+b_1=a_2+b_2\}|.
$$
\begin{lemma}\label{additiveenergy}
Let $A, B\subseteq M_n(\mathbb{F}_q)$ and $C\subseteq GL_n{\mathbb{F}_q}$. We have
$$
E_+(A, B)\lesssim\frac{|BC|^2|A|^2}{q^{n^2}}+q^{2n^2-\frac{n+1}{2}}\frac{|BC||A|}{|C|}.
$$
\end{lemma}
\begin{proof}
By definition, we have
\begin{equation}
  \begin{split}
    E_+(A, B) & =|\{(a_1, a_2, b_1, b_2)\in A^2\times B^2:a_1+b_1=a_2+b_2\}| \\
      & =|C|^{-2}|\{(a_1, a_2, b_1, b_2, c_1, c_2)\in A^2\times B^2\times C^2:a_1+b_1c_1c_1^{-1}=a_2+b_2c_2c_2^{-1}\}|\\
      &\leq|C|^{-2}|\{(a_1, a_2, s_1, s_2, t_1, t_2)\in A^2\times (BC)^2\times (C^{-1})^2:a_1+s_1t_1=a_2+s_2t_2\}|\\
      &=|C|^{-2}N(BC, C^{-1}, A, -A, BC, C^{-1}).
  \end{split}
\end{equation}
It follows from Proposition \ref{main2} that
\begin{equation}
  \begin{split}
    E_+(A, B) &\leq|C|^{-2}N(BC, C^{-1}, A, -A, BC, C^{-1})\\
    &\lesssim|C|^{-2}\left(\frac{|BC|^2|C|^2|A|^2}{q^{n^2}}+q^{2n^2-\frac{n+1}{2}}|BC||C||A|\right)\\
    &=\frac{|BC|^2|A|^2}{q^{n^2}}+q^{2n^2-\frac{n+1}{2}}\frac{|BC||A|}{|C|}.
  \end{split}
\end{equation}
\end{proof}
For $\lambda\in A+B$, define
$$
t_{A+B}(\lambda)=|\{(a, b)\in A\times B:a+b=\lambda\}|.
$$
By the Cauchy-Schwarz inequality, we have
$$
(|A||B|)^2=(\sum_{\lambda\in A+B}t_{A+B}(\lambda))^2\leq|A+B|\sum_{\lambda\in A+B}t_{A+B}(\lambda)^2=|A+B|E_+(A, B).
$$

Now we are able to prove Theorem \ref{A+AAA2}. We also restate Theorem \ref{A+AAA2} here.
\begin{theorem}
For every positive integer $n$, there exists $C(n)$ such that the following holds. If $A\subseteq M_n(\mathbb{F}_q)$ with $|A|\geq C(n)q^{n^2-1}$, we have
$$
\max\{|A+A|, |AA|\}\gtrsim \min\left\{\frac{|A|^2}{q^{n^2-\frac{n+1}{4}}}, q^{n^2/3}|A|^{2/3}\right\}.
$$
\end{theorem}
\begin{proof}
Since $|A|\geq C(n)q^{n^2-1}$ and $|Z_n(\mathbb{F}_q)|\sim q^{n^2-1}$, we choose $C(n)$ such that $|A|>2|Z_n(\mathbb{F}_q)|$. Then $|A\cap GL_n(\mathbb{F}_q)|\geq |A|/2$. And hence we can assume that $A\subseteq GL_n(\mathbb{F}_q)$. Applying Lemma \ref{additiveenergy} with $A=B=C$, we have
\begin{equation}
  \begin{split}
    \frac{|A|^4}{|A+A|} & \leq E_+(A, A)\\
      & \lesssim \frac{|AA|^2|A|^2}{q^{n^2}}+q^{2n^2-\frac{n+1}{2}}|AA|.
  \end{split}
\end{equation}
Therefore
$$
\max\{|A+A|, |AA|\}\gtrsim \min\left\{\frac{|A|^2}{q^{n^2-\frac{n+1}{4}}}, q^{n^2/3}|A|^{2/3}\right\}.
$$

\end{proof}

Furthermore, we have another theorem, which also generalizes the result in  \cite{2021arXiv210607328M}.
\begin{theorem}
Let $A, B, C, D\subseteq M_n(\mathbb{F}_q)$, and let $N$ denote the number of solutions to the equation
$$
a+b=cd,\quad (a, b, c, d)\in A\times B\times C\times D.
$$
Then we have
$$
N\lesssim\frac{|A||B|^{\frac{1}{2}}|C||D|}{q^{\frac{n^2}{2}}}+q^{n^2-\frac{n+1}{4}}(|A||C||D||B|)^{\frac{1}{2}}.
$$
\end{theorem}
\begin{proof}
For every $b\in B$, let
$$
r(b)=|\{(a, c, d)\in A\times C\times D:-a+cd=b\}|.
$$
By definition, we have $N=\sum_{b\in B}r(b)$. The Cauchy-Schwarz inequality implies that
$$
N^2=\left(\sum_{b\in B}r(b)\right)^2\leq|B|\sum_{b\in B}r(b)^2.
$$
Note that
\begin{equation*}
\begin{split}
  \sum_{b\in B}r(b)^2 & = |\{(a_1, c_1, d_1,a_2, c_2, d_2)\in A\times C\times D\times A\times C\times D:-a_1+c_1d_1=-a_2+c_2d_2\in B\}|\\
    & \leq|\{(a_1, c_1, d_1,a_2, c_2, d_2)\in A\times C\times D\times A\times C\times D:-a_1+c_1d_1=-a_2+c_2d_2\}|\\
    &=N(C, D, A, -A, -C, D)\\
    &\lesssim \frac{|A|^2|C|^2|D|^2}{q^{n^2}}+q^{2n^2-\frac{n+1}{2}}|A||C||D|.
\end{split}
\end{equation*}
So
$$
N\lesssim\frac{|A||B|^{\frac{1}{2}}|C||D|}{q^{\frac{n^2}{2}}}+q^{n^2-\frac{n+1}{4}}(|A||C||D||B|)^{\frac{1}{2}}.
$$

\end{proof}

\section{Proof of Theorem \ref{(A+B)C}}
In this section, we give another proof of Theorem \ref{(A+B)C}.

We construct a graph $G'=(U'\cup V', E')$, where $U'=V'=(M_n(\mathbb{F}_q))^3$. There is an edge between $(a, e, c)\in U$ and $(b, f, d)\in V$ if and only if $ba+ef=c+d$. The only difference here compared to the graph in Section \ref{keylemma} is that we switch between $ba$ and $ab$. We still have $$
|U'|=|V'|=(|M_n(\mathbb{F}_q)|)^3=q^{3n^2},\ \deg(U')=q^{2n^2},\ \text{ and }\frac{\deg(U')}{|V'|}=\frac{1}{q^{n^2}}.
$$

For any two points $(a_1, e_1, c_1)$ and $(a_2, e_2, c_2)$ in $U'$, we count the number of their common neighbors, i.e., the number of solutions $(b, f, d)$ to the equations
\begin{equation}\label{11}
  ba_1+e_1f=c_1+d, \quad ba_2+e_2f=c_2+d.
\end{equation}
So we have
\begin{equation}\label{12}
b(a_1-a_2)+(e_1-e_2)f=c_1-c_2.
\end{equation}
A solution $(b, f)$ to equation (\ref{12}) corresponds to a solution $(b, f, ba_1+e_1f-c_1)$ to equations (\ref{11}). So we only need to determine the number of solutions to equation (\ref{12}).

Let $k_1=\mathrm{rank}(e_1-e_2)$ and $k_2=\mathrm{rank}(a_1-a_2)$.
Then there exist $P_1, Q_1, P_2, Q_2\in GL_n(\mathbb{F}_q)$, such that $P_1(e_1-e_2)Q_1=
\left(
\begin{array}{cc}
I_{k_1} & 0 \\
0 & 0 \\
\end{array}
\right)
$, and $P_2(a_1-a_2)Q_2=\left(
\begin{array}{cc}
I_{k_2} & 0 \\
0 & 0 \\
\end{array}
\right)$.
Equation (\ref{12}) becomes
\begin{equation}\label{13}
  P_1bP_2^{-1}P_2(a_1-a_2)Q_2+P_1(e_1-e_2)Q_1Q_1^{-1}fQ_2=P_1(c_1-c_2)Q_2,
\end{equation}
i.e.,
\begin{equation}\label{14}
  P_1bP_2^{-1}\left(
  \begin{array}{cc}
  I_{k_2} & 0 \\
  0 & 0 \\
  \end{array}
  \right)+\left(
  \begin{array}{cc}
  I_{k_1} & 0 \\
  0 & 0 \\
  \end{array}\right)Q_1^{-1}fQ_2=P_1(c_1-c_2)Q_2.
\end{equation}
If we write $b'=P_1bP_2^{-1}$ and $f'=Q_1^{-1}fQ_2$, then a solution $(b, f)$ to equation (\ref{12}) corresponds to a solution $(b', f')$ to equation
\begin{equation}\label{15}
  b'\left(
  \begin{array}{cc}
  I_{k_2} & 0 \\
  0 & 0 \\
  \end{array}
  \right)+\left(
  \begin{array}{cc}
  I_{k_1} & 0 \\
  0 & 0 \\
  \end{array}\right)f'=P_1(c_1-c_2)Q_2.
\end{equation}
So the number of solutions $(b, f)$ to equation (\ref{12}) is equal to the number of solutions $(b', f')$ to equation (\ref{15}). If we write $b'=(b_{ij})_{1\leq i,j\leq n}$, $f'=(f_{ij})_{1\leq i,j\leq n}$, and $P_1(c_1-c_2)Q_2=(c_{ij})_{1\leq i,j\leq n}$, then equation (\ref{15}) becomes
\begin{equation}\label{16}
  \left\{
\begin{array}{lllll}
  b_{ij}+f_{ij}=c_{ij}, & \text{ for }1\leq i\leq k_1, \text{ and }1\leq j\leq k_2;\\
  f_{ij}=c_{ij}, & \text{ for }1\leq i\leq k_1, \text{ and }k_2+1\leq j\leq n;\\
  b_{ij}=c_{ij},&\text{ for }k_1+1\leq i\leq n, \text{ and }1\leq j\leq k_2;\\
  c_{ij}=0, & \text{ for }k_1+1\leq i\leq n, \text{ and }k_2+1\leq j\leq n.
\end{array}
  \right.
\end{equation}
Therefore, equation (\ref{12}) has a solution if and only if $c_{ij}=0$ for $k_1+1\leq i\leq n$ and $k_2+1\leq j\leq n$.
And if equation (\ref{12}) has a solution, then it is not difficult to calculate that the number of solutions is $q^{2n^2-k_1n-k_2n+k_1k_2}$.

If $k_1=k_2=0$, then $c_{ij}=0$ for $1\leq i\leq n$ and $1\leq j\leq n$, i.e., $P_1(c_1-c_2)Q_2=0$. It follows that $a_1=a_2, e_1=e_2, c_1=c_2$, which contradicts that $(a_1, e_1, c_1)$ and $(a_2, e_2, c_2)$ are different. So equation (\ref{12}) has no solution if $k_1=k_2=0$. If either $k_1$ or $k_2$ is equal to $n$, without loss of generality, assuming that $k_1=n$, then equation (\ref{12}) always has a solution $(b, f)$ where $f=(e_1-e_2)^{-1}(c_1-c_2-b(a_1-a_2))$.

For $0\leq k_1, k_2\leq n$ (except the case that $k_1=k_2=0$), let $E_{k_1, k_2}$ be the adjacency matrix of the graph $G_{k_1, k_2}$, whose vertex set is $(M_n(\mathbb{F}_q))^3$ such that two vertices $(a_1, e_1, c_1)$ and $(a_2, e_2, c_2)$ form an edge in $G_{k_1, k_2}$ if and only if $\mathrm{rank}(e_1-e_2)=k_1$, $\mathrm{rank}(a_1-a_2)=k_2$, and equation (\ref{12}) has a solution. If $(0,0, 0)$ is adjacent with $(a, e, c)$, then $(a', e', c')$ is adjacent with $(a+a', e+e', c+c')$, and vice versa. So $G_{k_1, k_2}$ is regular. We count the degree of $(0,0, 0)$, i.e., the number of $(a, e, c)$ with the property that $\mathrm{rank}(e)=k_1$, $\mathrm{rank}(a)=k_2$, and $ba+ef=c$ has a solution.

We first choose $a$ and $e$ such that $\mathrm{rank}(e)=k_1$, $\mathrm{rank}(a)=k_2$. Theorem \ref{rank} implies that there are $\frac{Q_{k_2}(q^{n})Q_{k_2}(q^n)}{Q_{k_2}(q^{k_2})}\frac{Q_{k_1}(q^{n})Q_{k_1}(q^n)}{Q_{k_1}(q^{k_1})}$ choices for $(a, e)$. Next we choose $c$. Given $a$ and $e$ such that $\mathrm{rank}(e)=k_1$ and $\mathrm{rank}(a)=k_2$, there exist $P_1, Q_1, P_2, Q_2\in GL_n(\mathbb{F}_q)$, such that $P_1eQ_1=
\left(
\begin{array}{cc}
I_{k_1} & 0 \\
0 & 0 \\
\end{array}
\right)
$, and $P_2aQ_2=\left(
\begin{array}{cc}
I_{k_2} & 0 \\
0 & 0 \\
\end{array}
\right)$.
Equation (\ref{16}) implies that the only restrictions to $c$ are $(P_1cQ_2)_{ij}=0$ for $k_1+1\leq i\leq n$ and $k_2+1\leq j\leq n$. There are $q^{n^2-(n-k_1)(n-k_2)}=q^{nk_1+nk_2-k_1k_2}$ choices for $P_1cQ_2$. And hence there are $q^{nk_1+nk_2-k_1k_2}$ choices for $c$. Thus the number of $(a, e, c)$ with the property that $\mathrm{rank}(e)=k_1$, $\mathrm{rank}(a)=k_2$, and $ba+ef=c$ has a solution is
$$
\frac{Q_{k_2}(q^{n})Q_{k_2}(q^n)}{Q_{k_2}(q^{k_2})}\frac{Q_{k_1}(q^{n})Q_{k_1}(q^n)}{Q_{k_1}(q^{k_1})}q^{nk_1+nk_2-k_1k_2}\sim q^{3nk_1+3nk_2-k_1^2-k_2^2-k_1k_2}.
$$

For $0\leq k_1, k_2\leq n-1$, let $F_{k_1, k_2}$ be the adjacency matrix of the graph $H_{k_1, k_2}$, whose vertex set is $(M_n(\mathbb{F}_q))^3$, such that two vertices $(a_1, e_1, c_1)$ and $(a_2, e_2, c_2)$ form an edge in $H_{k_1, k_2}$  if  and only if $\mathrm{rank}(e_1-e_2)=k_1$, $\mathrm{rank}(a_1-a_2)=k_2$, and equation (\ref{12}) has no solution. If $(0,0, 0)$ is adjacent with $(a, e, c)$, then $(a', e', c')$ is adjacent with $(a+a', e+e', c+c')$, and vice versa. So $H_{k_1, k_2}$ is regular. We count the degree of $(0,0, 0)$, i.e., the number of $(a, e, c)$ with the property that $\mathrm{rank}(e)=k_1$, $\mathrm{rank}(a)=k_2$, and $ba+ef=c$ has no solution.

We first choose $a$ and $e$ such that $\mathrm{rank}(e)=k_1$, $\mathrm{rank}(a)=k_2$. Theorem \ref{rank} implies that there are $\frac{Q_{k_2}(q^{n})Q_{k_2}(q^n)}{Q_{k_2}(q^{k_2})}\frac{Q_{k_1}(q^{n})Q_{k_1}(q^n)}{Q_{k_1}(q^{k_1})}$ choices for $(a, e)$. Next we choose $c$. According to the above argument, there are $q^{n^2}-q^{nk_1+nk_2-k_1k_2}$ choices for $c$. Thus the number of $(a, e, c)$ with the property that $\mathrm{rank}(e)=k_1$, $\mathrm{rank}(a)=k_2$, and $ba+ef=c$ has no solution is
$$
\frac{Q_{k_2}(q^{n})Q_{k_2}(q^n)}{Q_{k_2}(q^{k_2})}\frac{Q_{k_1}(q^{n})Q_{k_1}(q^n)}{Q_{k_1}(q^{k_1})}(q^{n^2}-q^{nk_1+nk_2-k_1k_2})\sim q^{n^2+2nk_1+2nk_2-k_1^2-k_2^2}.
$$

Based on the previous calculation, we have
\begin{equation}\label{17}
\begin{split}
  NN^T  =& q^{n^2}J+(\deg(U)-q^{n^2})I+\sum_{k_2=1}^{n}(q^{2n^2-k_2n}-q^{n^2})E_{0, k_2}\\
  &+\sum_{k_2=0}^{n}\sum_{k_1=1}^{n}(q^{2n^2-k_1n-k_2n+k_1k_2}-q^{n^2})E_{k_1, k_2}-\sum_{k_1, k_2=0}^{n-1}q^{n^2}F_{k_1, k_2}\\
  =&q^{n^2}J+(\deg(U)-q^{n^2})I+\sum_{k_2=1}^{n-1}(q^{2n^2-k_2n}-q^{n^2})E_{0, k_2}\\
  &+\sum_{k_2=0}^{n-1}\sum_{k_1=1}^{n-1}(q^{2n^2-k_1n-k_2n+k_1k_2}-q^{n^2})E_{k_1, k_2}-\sum_{k_1, k_2=0}^{n-1}q^{n^2}F_{k_1, k_2}.
\end{split}
\end{equation}

Let $v^3 = (u_1, \ldots , u_{|U'|}, v_1, \ldots , v_{|V'|})^T$ be an eigenvector of $A_{G'}$ corresponding to the eigenvalue $\lambda_3$. Lemma \ref{NNT} implies that $(u_1, \ldots , u_{|U'|})^T$ is an eigenvector of $NN^T$ corresponding to the eigenvalue $\lambda_3^2$. It follows from equation (\ref{9}) that
\begin{equation}\label{18}
\begin{split}
    &(\lambda_3^2-\deg(U)+q^{n^2})(u_1, \ldots , u_{|U'|})^T \\
    =& \left(\sum_{k_2=1}^{n-1}(q^{2n^2-k_2n}-q^{n^2})E_{0, k_2}
  +\sum_{k_2=0}^{n-1}\sum_{k_1=1}^{n-1}(q^{2n^2-k_1n-k_2n+k_1k_2}-q^{n^2})E_{k_1, k_2}-\sum_{k_1, k_2=0}^{n-1}q^{n^2}F_{k_1, k_2}\right)(u_1, \ldots , u_{|U'|})^T.
\end{split}
\end{equation}
Therefore, $(u_1, \ldots , u_{|U'|})^T$ is an eigenvector of
$$
\sum_{k_2=1}^{n-1}(q^{2n^2-k_2n}-q^{n^2})E_{0, k_2}
  +\sum_{k_2=0}^{n-1}\sum_{k_1=1}^{n-1}(q^{2n^2-k_1n-k_2n+k_1k_2}-q^{n^2})E_{k_1, k_2}-\sum_{k_1, k_2=0}^{n-1}q^{n^2}F_{k_1, k_2}
$$
corresponding to the eigenvalue $\lambda_3^2-\deg(U')+q^{n^2}$.

Since $G_{k_1, k_2}$ is regular, for every eigenvalue $\lambda$ of $E_{k_1, k_2}$, we have $|\lambda|\lesssim q^{3nk_1+3nk_2-k_1^2-k_2^2-k_1k_2}$. Since $H_{k_1, k_2}$ is regular, for every eigenvalue $\lambda$ of $F_{k_1, k_2}$, we have $|\lambda|\lesssim q^{n^2+2nk_1+2nk_2-k_1^2-k_2^2}$. So if $\lambda$ is an eigenvalue of
$$
\sum_{k_2=1}^{n-1}(q^{2n^2-k_2n}-q^{n^2})E_{0, k_2}
  +\sum_{k_2=0}^{n-1}\sum_{k_1=1}^{n-1}(q^{2n^2-k_1n-k_2n+k_1k_2}-q^{n^2})E_{k_1, k_2}-\sum_{k_1, k_2=0}^{n-1}q^{n^2}F_{k_1, k_2},
$$
then
\begin{equation}
\begin{split}
|\lambda|\lesssim&\sum_{k_2=1}^{n-1}(q^{2n^2-k_2n}-q^{n^2})q^{3nk_2-k_2^2}
  +\sum_{k_2=0}^{n-1}\sum_{k_1=1}^{n-1}(q^{2n^2-k_1n-k_2n+k_1k_2}-q^{n^2})q^{3nk_1+3nk_2-k_1^2-k_2^2-k_1k_2}\\
 & +\sum_{k_1, k_2=0}^{n-1}q^{n^2}q^{n^2+2nk_1+2nk_2-k_1^2-k_2^2}\\
\leq&\sum_{k_2=1}^{n-1}q^{2n^2+2nk_2-k_2^2}
  +\sum_{k_2=0}^{n-1}\sum_{k_1=1}^{n-1}q^{2n^2+2nk_1+2nk_2-k_1^2-k_2^2}+\sum_{k_1, k_2=0}^{n-1}q^{2n^2+2nk_1+2nk_2-k_1^2-k_2^2}\\
\lesssim&\sum_{k_1, k_2=0}^{n-1}q^{2n^2+2nk_1+2nk_2-k_1^2-k_2^2}.
\end{split}
\end{equation}

Given $k_1$ and $n$, observe that the function $g(k_2)=2n^2+2nk_1+2nk_2-k_1^2-k_2^2$ is increasing for $k_2\leq n$, so the maximum occurs at $k_2=n-1$ and $2n^2+2nk_1+2nk_2-k_1^2-k_2^2\leq2n^2+2nk_1+2n(n-1)-k_1^2-(n-1)^2=3n^2+2nk_1-k_1^2-1$. Similarly, $3n^2+2nk_1-k_1^2-1$ attains maximum at $k_1=n-1$. Thus $3n^2+2nk_1-k_1^2-1\leq3n^2+2n(n-1)-(n-1)^2-1=4n^2-2$.
Therefore, the eigenvalue $\lambda_3^2-\deg(U')+q^{n^2}$ of
$$
\sum_{k_2=1}^{n-1}(q^{2n^2-k_2n}-q^{n^2})E_{0, k_2}
  +\sum_{k_2=0}^{n-1}\sum_{k_1=1}^{n-1}(q^{2n^2-k_1n-k_2n+k_1k_2}-q^{n^2})E_{k_1, k_2}-\sum_{k_1, k_2=0}^{n-1}q^{n^2}F_{k_1, k_2}
$$
satisfies that
$$
|\lambda_3^2-\deg(U')+q^{n^2}|\lesssim q^{4n^2-2}.
$$
Note that $\deg(U')=q^{2n^2}$. So we conclude that
$$
|\lambda_3|\lesssim q^{2n^2-1}.
$$

Now if $A, B\subseteq M_n(\mathbb{F}_q)$ and $C\subseteq GL_n(\mathbb{F}_q)$, then we put $X=\{(c_1, -b_2, -a_1c_1): a_1\in A, b_2\in B, c_1\in C\}\subseteq U'$ and $Y=\{(b_1, c_2, a_2c_2):a_2\in A, b_1\in B, c_2\in C\}\subseteq V'$. Since $C\subseteq GL_n(\mathbb{F}_q)$, we have $|X|=|Y|=|A||B||C|$. Note that the number of edges between $X$ and $Y$ is equal to
$$
|\{(a_1, b_1, c_1, a_2, b_2, c_2)\in A\times B\times C\times A\times B\times C:(a_1+b_1)c_1=(a_2+b_2)c_2\}|.
$$
Similar to the proof of Theorem \ref{A+BC3}, we have
\begin{equation*}
  \begin{split}
    \frac{|A|^2|B|^2|C|^2}{|(A+B)C|} & \leq|\{(a_1, b_1, c_1, a_2, b_2, c_2)\in A\times B\times C\times A\times B\times C:(a_1+b_1)c_1=(a_2+b_2)c_2\}| \\
      & =e(X, Y)\\
      &\lesssim\frac{\deg(U')}{|V'|}|X||Y|+|\lambda_3|\sqrt{|X||Y|}\\
      &\lesssim\frac{|A|^2|B|^2|C|^2}{q^{n^2}}+q^{2n^2-1}|A||B||C|.
  \end{split}
\end{equation*}
Therefore,
$$
|(A+B)C|\gtrsim\min\left\{q^{n^2}, \frac{|A||B||C|}{q^{2n^2-1}}\right\}.
$$

\bibliographystyle{abbrv}
\bibliography{REF}
\end{document}